\documentclass{jktr}

\usepackage{float} % for H option of images

\newcommand*{\QEDA}{\hfill\ensuremath{\square}}%

\usepackage{xcolor}

\usepackage{verbatim}
\usepackage{float} % for H option of images

\usepackage{enumitem} % for (i), (ii) in enumerate

\usepackage{afterpage} % used to remove page number from specific page

\usepackage{subfigure}

\newcommand{\Z}{\mathbb{Z}}

\newcommand{\R}{\mathbb{R}}

\newcommand{\cR}{\mathcal{R}}

\newcommand{\eps}{\varepsilon}

\renewcommand{\d}{\partial}

\newcommand{\sign}{\operatorname{sign}}

\renewcommand{\Im}{\operatorname{Im}}
\renewcommand{\int}{\operatorname{int}}

\renewcommand{\v}{\mathbf{v}}

\newcommand{\bo}{\mathbf{0}}
\newcommand{\0}{\{0\}}

\renewcommand{\t}{\{t\}}

\renewcommand{\mod}{\text{ mod }}

\newcommand{\tw}{\widetilde}
\newcommand{\tl}{\widetilde{\lambda}}

\newcommand{\tY}{\tw{Y}}
\newcommand{\td}{\tw{\d}}

\newcommand{\te}{\tw{e}}
\newcommand{\tA}{\tw{A}}

\newcommand{\stimes}{\hskip-0.1cm\times\hskip-0.1cm}

\renewcommand{\theequation}{\arabic{equation}}
\newcounter{labelflag} 

\newcommand{\Label}[1]{
                       \ifnum\thelabelflag=1
                          \ifmmode
                             \makebox[0in][l]{\qquad\fbox{\rm#1}}
                          \else
                             \marginpar{\vspace{0.7\baselineskip}
                                        \hspace{-1.1\textwidth}
                                        \fbox{\rm#1}}
                          \fi
                       \fi
                       \label{#1}
                      }
\newcommand{\ba}{\begin{align}}
\newcommand{\ea}{\end{align}}

\newlength{\jump}

\newcounter{mysub}
\makeatletter
\newcommand{\manuallabel}[2]{\def\@currentlabel{#2}\label{#1}}
\makeatother

\newcommand{\mycomment}[1]{}%   % allow nested comments.

\begin{document}

\markboth{Ash Lightfoot}
{The Schneiderman-Teichner invariant applied to immersions arising from link maps in $S^4$ }

%%%%%%%%%%%%%%%%%%%%% Publisher's Area please ignore %%%%%%%%%%%%%%
\catchline{}{}{}{}{}
%%%%%%%%%%%%%%%%%%%%%%%%%%%%%%%%%%%%%%%%%%%%%%%%%%%%%%%%%%%%%%%%%%%

\title{The Schneiderman-Teichner invariant applied to immersions arising from link maps in $S^4$ }

\author{Ash Lightfoot}

\address{Indiana University\\ Rawles Hall \\  Bloomington IN 47405}

\maketitle

\begin{abstract}
It is shown that Schneiderman and Teichner's invariant $\tau$ detects the example due to Kirk of a link map $f:S^2_+\cup S^2_-\to S^4$ for which  the restricted map $f|S^2_+:S^2_+\to S^4-f(S^2_-)$ has vanishing Wall self-intersection but is not homotopic to an embedding.
% In doing so we show how a ``movie'' of a link map  yields a handlebody decomposition for the complement in $S^4$ of one of the components, and how this can be used to compute the (INT) relations defined in \cite{ST} for $\tau$.
\end{abstract}



% Reduce space before and after figures
\addtolength{\intextsep}{-0.4cm}
%\captionsetup[sub]{font=normal, labelformat=\thefigure(\thesubfigure)}
{\section{Introduction}
A link map  is a map from a union of spheres into another sphere with pairwise disjoint images. The equivalence relation placed on the set of link maps is that of link homotopy: two link maps are link homotopic if they are homotopic through link maps. We restrict our attention to link maps of the form $S^2\cup S^2\to S^4$, whose study was initiated when Fenn and Rolfsen (\cite{FR}) produced the first non-trivial example (that is, a link map that is not link homotopic to the trivial link)  and a link homotopy invariant to detect it. Their idea was generalized by Kirk (\cite{Ki1}, \cite{Ki2}) to define his $\sigma$-invariant, which takes values in $\Z[t]\oplus \Z[t]$ and gives necessary conditions for a link map to be equivalent to the trivial link. It is still an open question whether $\sigma$ is a complete obstruction. (The counterexamples constructed in \cite{Li} were found to be in error by Pilz (\cite{Pilz}). The author is grateful to Rob Schneiderman for pointing this out and  to Uwe Kaiser for supplying a copy of \cite{Pilz}.)

Consider a link map  $f:S^2_+\cup S^2_-\to S^4$ with $\sigma_+(f)=0$, and suppose furthermore that each component $f^{\pm}:=f|{S^2_{\pm}}: S^2_{\pm}\to S^4-f(S^2_{\mp})$ is self-transverse and $\pi_1(S^4-f(S^2_-))\cong \Z$. The double points of $f^+$ may then be equipped with (framed, immersed) Whitney disks in $S^4-f(S^2_-)$. Li (\cite{Li}) defined a $\Z_2$-valued homotopy invariant $\omega_+(f)$ of $f^+$ that counts weighted intersections between $\Im f^{\pm}$ and these Whitney disks so as to obstruct homotoping $f^+$ to an embedding.  On the other hand, in \cite{ST}  Schneiderman and Teichner introduced a homotopy invariant $\tau$ that takes as input a more general map $g:S^2\to X^4$ %, of a 2-sphere into a general 4-manifold $X$,
that has vanishing Wall obstruction $\mu(g)$ (c.f. \cite{W}), and obstructs homotoping $g$ to an embedding.  It takes values in a quotient of the group $\Z[\pi_1(X)\times \pi_1(X)]$, and when applied to the above situation $\tau(f^{+})$ counts the same intersections as $\omega_{+}(f)$, but assigns to each intersection point a weight in the group ring $\Z[s^{\pm 1},t^{\pm 1}]$.

\mycomment{
The question of whether $\sigma$ is a complete obstruction was answered negatively by Li (\cite{Li}), who defined a $\Z_2\oplus\Z_2$-valued invariant $\omega$ of link maps with vanishing $\sigma$-invariant and used this to produce several different link homotopy
classes of such link maps.

The link homotopy invariant $\omega$ applied to a link map $f:S^2_+\cup S^2_-\to S^4$ has two components $\omega=(\omega_+,\omega_-)$, where $\omega_{\pm}$ measures the obstruction to homotoping the restricted map $f^{\pm}:=f|{S^2_{\pm}}: S^2_{\pm}\to S^4-f(S^2_{\mp})$ to an embedding. After performing a link homotopy of $f$ to ensure that $\pi_1(S^4-f(S^2_{\mp}))\cong \Z$), this is done by counting modulo 2 intersections between $\Im f^{\pm}$ and its Whitney discs, whose existence is guaranteed by the vanishing of $\sigma_{\pm}(f)$. In \cite{ST}, Schneiderman and Teichner introduced an invariant $\tau$ that takes as input an immersed 2-sphere in a general 4-manifold $X$ with vanishing Wall obstruction (see \cite{W} for details) and whose vanishing is a necessary condition to embedding. It takes values in a quotient of the group $\Z[\pi_1(X)\times \pi_1(X)]$. When applied to the above situation with $\pi_1(S^4-f(S^2_{\mp}))\cong \Z$, the invariant $\tau(f^{+})$ counts the same intersections as $\omega_{+}$, but assigns to each intersection point a weight in the group ring $\Z[s^{\pm 1},t^{\pm 1}]$.
}%
\mycomment{
{\color{blue}%
It is an open question whether $\sigma$ is a complete obstruction\footnote{\color{blue}The previous version of this paper cited (but did not use) an incorrect result of Li (\cite{Li}); namely, the construction of a link map $f$ with $\sigma(f)=(0,0)$ that is not link homotopic to an embedding. The error in that paper was corrected by Pilz\cite{Pilz} (Ch. 4). The author would like to thank Rob Schneiderman and Uwe Kaiser for, respectively, alerting him to the error and supplying a copy of \cite{Pilz}.}.
}

}
%The purpose of this paper is to compute $\tau$ in the context of a link
Link maps give rise to further examples of the failure of Whitney's trick in dimension 4 in the following way. In \cite{Ki2}, Kirk showed that the full $\sigma$-invariant is an obstruction to embedding each component and constructed a simple example of a link map $f$ with $\sigma(f)=(0,t^2-4t+3)$. Consequently, (after a small cusp homotopy) the restricted map $f^+:S^2_+\to S^4-f(S^2_-)$ has vanishing Wall obstruction $\mu(f^+)$ but is not homotopic to an embedding. We give a new proof of the latter by computing $\tau(f^+)\neq 0$.

\begin{theorem}\label{thm:main}
For the link map $f:S^2_+\cup S^2_-\to S^4$ constructed in \cite{Ki2} with $\sigma(f)=(0,t^2-4t+3)$, the restricted map $f^+:S^2_+\to S^4-f(S^2_-)$ has vanishing Wall self-intersection $\mu(f^+)$ but nonzero Schneiderman-Teichner invariant $\tau(f^+)$.
\end{theorem}

\section{Preliminaries}\label{sec:prelim}

In this section we recall the definitions of the relevant invariants as they pertain to our particular setting. Assume all manifolds are oriented. Let $f:S^2_+\cup S^2_-\to S^4$ be a link map whose restriction $f^{\pm}:S^2_{\pm}\to S^4-f(S^2_{\mp})$ to each component is a  self-transverse immersion.  Let $X_-=S^4-\nu f(S^2_-)$, and choose a generator $t$ for $H_1(X_-)\cong \Z$. For each double point $p$ of $f^+$, the image under $f^+$ of an arc in $S^2_+$ connecting the two preimages of $p$ is a loop representing $t^{n(p)}\in H_1(X_-)$ for some  $n(p)\in \Z$. The absolute value of $n(p)$ is well-defined, so one may define
\[
    \sigma_+(f) = \sum_{p} \sign(p) (t^{|n(p)|}-1)\in \Z[t],
\]
where the sum is over all double points of $f^+$. Reversing the roles of $f^+$ and $f^-$ one similarly defines $\sigma_-(f)$. The pair $\sigma(f)=(\sigma_+(f),\sigma_-(f))$ is referred to as the \emph{full} $\sigma$-invariant of $f$.

Suppose that $\pi_1(X)\cong \Z$, with generator $t$, and write $X=X_-$. After performing a number of local homotopies if necessary, $f^+$ may be assumed to have vanishing self-intersection number so that if $\sigma_+(f)=0$ then the double points of $f^+$ can be decomposed  into pairs $\{p_i^{\pm 1}\}_{i=1}^k$  such that $\sign(p_i^{\pm 1})=\pm 1$ and  $n(p_i^+)=n(p_i^-)$. Denote the two preimages of $p_i^{+}$ (resp. $p_i^{-}$) in $S^2_+$ by $x_i^{+}$ (resp. $x_i^{-}$) and $y_{i}^{+}$ (resp. $y_{i}^{-}$). Choose mutually disjoint embedded arcs in $S^2_+$ connecting $x_i^{+}$ to $y_i^{+}$ and $x_i^{+}$ to $y_i^{-}$. The image under $f^+$ of the union of these two arcs is a nulhomotopic loop $\gamma_i$ in $X$, which we call a \emph{Whitney circle} for the pair $\{p_i^+,p_i^-\}$. After a regular homotopy of $f^+$ the Whitney circles may be assumed to be embedded and mutually disjoint. Furthermore, it can be arranged that each Whitney circle $\gamma_i$ bounds an immersed 2-disc $W_i$ in $X$ whose interior is transverse to $f(S^2_+)$. We call $W_i$ a \emph{Whitney disc} for the pair $\{p_i^+,p_i^-\}$.

Choose a preferred arc $\alpha_i$ of $\d W_i\subset f(S^2_+)$  that runs between $p^+_i$ and $p^-_i$ and call this the \emph{positive arc} of $\d W_i=\gamma_i$. The arc $\beta_i$ of $\d W_i$ lying in the other sheet will be called the \emph{negative arc}. We will  refer to a neighborhood in $f(S^2_+)$ of $\alpha_i$ (resp. $\beta_i$)  as the \emph{positive} (resp. \emph{negative}) \emph{sheet} of $f(S^2_+)$ near $W_i$.

The loop based at $p_i^+$ that changes from the negative sheet to the positive sheet  determines a group element $t^{n_i}\in \pi_1(X)$, which we call the \emph{primary group element} for $W_i$ (here, as in elsewhere, we need not keep track of a basepoint for $X$ as $\pi_1(X)$ is abelian). In \cite{Li}, the non-negative integer $|n_i|=|n(p_i^{\pm})|$ is called the $n$-\emph{multiplicity} of $p_i^{\pm}$. Now, to each point $x\in \int W_i \cap f(S^2_+)$, the loop based at the basepoint of $f(S^2_+)$ that first goes along $f(S^2_+)$ to $x$, then along $W_i$ to the positive arc of $W_i$, then back to the basepoint of $f(S^2_+)$ along $f(S^2_+)$, determines the \emph{secondary group element} $t^{m_x}\in \pi_1(X)$ for $W_i$ corresponding to $x$. In \cite{Li}, the absolute value of the integer $m_x$ is called the $m$-\emph{multiplicity} of intersection point $x$.

Orient $\d W_i$ from $p_i^-$ to $p_i^+$ along the positive arc and back to $p_i^-$ along the negative arc. The positive tangent to $\d W_i$ together with an outward pointing second vector orient $W_i$.
Since the positive and negative sheets meet transversely at $p_i^\pm$, there are a pair of smooth vector fields $\v_1, \v_2$ on $\d W_i$ such that $\v_1$ is tangent to $f(S^2_+)$ along $\alpha_i$ and normal to $f(S^2_+)$ along $\beta_i$, while $\v_2$ is normal to $f(S^2_+)$ along $\alpha_i$ and tangent to $f(S^2_+)$ along $\beta_i$. Such a pair defines a normal framing of $W_i$ on the boundary. We say that $\{\v_1,\v_2\}$ is a \emph{correct framing} of $W_i$, and that $W_i$ is \emph{framed}, if the pair extends to a normal framing of $W_i$.

We use the notation $s^nt^m:=(t^n,t^m)$ so as to write $\Z[\pi_1(X)\times \pi_1(X)] = \Z[s^{\pm 1},t^{\pm 1}]$, and define
\[
    I(W_i)=\sum_{x\in \int W_i\cap f(S^+)} \sign(x) s^{n_i}t^{n_x} \in \Z[s^{\pm 1},t^{\pm 1}].
\]
Suppose that the Whitney discs $\{W_i\}_{i=1}^k$ are each correctly framed. This assumption along with all the above hypotheses imply that the  Schneiderman and Teichner invariant $\tau$ applied to $f^+$ is given by
\[
    \tau(f^+) = \sum_{i=1}^k I(W_i) \in \Z[s^{\pm 1},t^{\pm 1}]/\cR.
\]
The relations $\cR$ are additively generated by the equations
%\begin{group}
\numberwithin{equation}{section}
\begin{align}
    s^{k}t^{k} - s^{k} &= 0, \; \; k\in \Z,\label{reln1}\\
    s^{k}t^{l}+s^{-k}t^{l-k} &= 0, \; \; k,l\in \Z,\label{reln2}\\
    s^{k}t^{l}+s^{l}t^{k}&=0,  \; \; k,l\in \Z,\label{reln3}\\
    s^{k}\lambda(f(S^2_+),A)&=\omega_2(A)s^k,  \; \; k\in \Z,\, A\in \pi_2(X),\label{reln4}
\end{align}
%\end{group}
where $\lambda: \pi_2(X)\times \pi_2(X)\to \Z[t^{\pm 1}]$ is the Wall intersection form (\cite{W}) and $\omega_2:\pi_2(X)\to \Z_2$ is the second Stiefel-Whitney class. By Theorem 2 of \cite{ST}, if $f^+$ is homotopic to an embedding then $\tau(f^+)=0$.

Let $1,t$ denote generators for $\Z_2\times \Z_2$, written multiplicatively, and write $\bar{n}:= n\mod 2$. Define a ring homomorphism $\Phi: \Z[\Z\times \Z] \to \Z_2[\Z_2\times \Z_2]$ by sending
\[
    as^kt^l\mapsto \begin{cases} \bar{a} \; \text{ if $k=0=l\text{ mod }2,$}\\
    \bar{a}t \; \text{ otherwise,}    \end{cases}
\]
for any $a, k, l\in \Z$,  and extending by linearity. Note that for a monomial $at^k \in \Z[t^{\pm 1}]$, we have $\Phi(at^k)=\bar{a}t^{\bar{k}}$.

For our example  we will see that $\Phi$ induces a well-defined homomorphism $\Z[s^{\pm 1},t^{\pm 1}]/\cR\to \Z_2[\Z_2\times \Z_2]$.

%{\color{red}
\begin{remark}
 We may equivalently define $\Phi(as^nt^m) = \bar{a}t^{\overline{n+nm+m}}$. Thus, if $\varphi:\Z_2[\Z_2\times \Z_2]\to \Z_2$ is defined by $t\mapsto 1$, $1\mapsto 0$, then $\varphi\circ \Phi(I(W_i)) \in \Z_2$ equals that which Li refers to as $I(W_i)$ in \cite{Li}.  It follows that Li's $\omega_+$-invariant may be written
\[
    \omega_+(f) = \varphi\circ \Phi(\tau(f^+))\in \Z_2.
\]
\begin{comment} by $\widehat{I}(W_i):= \Phi(I(W_i))(1)\in \Z_2$ is zero if the $n$-multiplicity $|n_i|$ of $W_i$ is even, while if $|n_i|$ is odd then $\widehat{I}(W_i)$ equals the modulo 2 number of intersection points $x\in \int W_i\cap f(S^2_+)$ with $m$-multiplicity $|m_x|$ odd.
\end{comment}
\end{remark}
%}

\section{The Example}\label{sec:example}

We recall the ``moving picture'' method  to construct link maps  in $S^4$ (c.f. \cite{Ki2}, \cite{Li}). One gives a sequence of pictures to illustrate a regular link homotopy in $\R^3$ from the 2-component unlink to itself. A link map $f:S^2_+\cup S^2_-\to S^4$ for which each component is immersed  is constructed by taking the trace of this homotopy and ``capping off'' with two pairs of embedded 2-discs. Using this method, Figs.  \ref{fig:movie}.1-\ref{fig:movie-last} in Section \ref{sec:figs} (where, for brevity, we interpret Figs. \ref{fig:movie}.1-\ref{fig:movie}.4 as depicting four 2-component links, each with one component the same dotted circle) describe a link map $f$ for which the restricted map $f^+:=f|{S^2_+}$ has 5 pairs of double points, paired with opposite signs, such that the first pair is of $n$-multiplicity $0$ and the latter 4 pairs are of $n$-multiplicity $1$. On the other hand, the map $f^-:=f|{S^2_-}$ has 4 doubles points of the same sign and with $n$-multiplicity 1, and one double point (of the opposite sign) with $n$-multiplicity $2$. For our purposes the actual sign of each double point does not need to be specified. It follows that $\sigma(f)=(0,t^2-4t+3)$ (up to sign) and, since $f^+$ has zero self-intersection number, $\mu(f^+)=0$ (c.f. \cite{Ki2} Section 6). To show that $\tau(f^+)\neq 0$ we make this description more explicit so that, using the methods of Section 6.2 in \cite{GS}, we may obtain a handlebody decomposition of $X=S^4- \nu f(S^2_-)$ and a description of $f(S^2_+)$ relative to this decomposition.

In each picture of Figs. \ref{fig:movie}.1-\ref{fig:movie-last} (which we henceforth refer to simply as Fig. \ref{fig:movie}), denote the dotted component and the other component, which we view as ``level circles'' of the components $f(S^2_-)$ and $f(S^2_+)$, by $C^-$ and $C^+$, respectively. Notice that if one removes from $f(S^2_-)$ the 2-disc bound by $C^-$ in Fig. \ref{fig:movie}.1 used in capping off, and let $F^-$ denote the resulting immersed 2-disc in $S^4-f(S^2_+)$, we have $X = D^4 - \nu F^-$. Identify $D^4 = [-1,2]\times D^3$. As an outline of what follows, the immersed 2-sphere $F^{+}:=f(S^2_+)\subset (-1,2)\times D^3$ is constructed by attaching a 0-handle and a 2-handle to the trace in $[0,1]\times D^3$ of a regular homotopy of an unknot (pictured in Fig. \ref{fig:movie} as family of undotted components). The properly immersed 2-disc $F^{-}\subset (-1,2]\times D^3$ is constructed by attaching to a 0-handle the trace in $[0,1]\times D^3$ of a regular homotopy of an unknot (the family of dotted components in Fig. \ref{fig:movie}. This is done in such a way that $F^+$ and $F^-$ are disjoint in $[-1,2]\times D^3$.

For each $t\in [-1,2]$, let $X_t=X\cap ([-1,t]\times D^3)$, $F^{\pm}_t = F^{\pm}\cap ([-1,t]\times D^3)$ and let $\d_t X = X \cap (\{t\}\times D^3)$; by the construction of $X$ and $F^{-}$ we shall have $X_t=([-1,t]\times D^3) - \nu F^-_t$.

For $t\in [0,1]$, the (perhaps singular) knots $\d F^{\pm}_t=F^{\pm}\cap (\{t\}\times D^3)$ appear in Fig. \ref{fig:movie} for various values of $t$ as the disjoint knots $C^{\pm}$, respectively. We take $F_0^{+}$ and $F_0^{-}$ to be disjoint, properly embedded 2-discs in $(-1,0]\times D^3$ with boundaries the unknots in $\0\times D^3$ appearing as the components $C^+$ and $C^-$ in Fig. \ref{fig:movie}.1, respectively. Then $F_0^{+}$ is properly embedded  in $X_0=([-1,0]\times D^3)-\nu(F_0^-)$, which is the 1-handlebody whose Kirby diagram is given by Fig. \ref{fig:movie}.1 (if one ignores the component $C^+$ and interprets the dotted circle in the usual way).

The link homotopy of the unlink $C^+\cup C^-$ to itself appearing in Figs. \ref{fig:movie}.1-\ref{fig:movie-last} decomposes into a sequence of regular link homotopies alternating between the following two types:\vspace{-0.2cm}
\begin{enumerate}[label=(\Roman{*}), ref=(\Roman{*})]
    \item $C^+$ undergoes a regular homotopy in the complement of $C^-$ (which remains fixed), or \label{it:reg-plus}
\begin{comment}
%% C^- isotopy - don't need for this example
\item $C^-$ undergoes an isotopy in the complement of $C^+$ (which remains fixed), or\label{it:iso-neg}
    \end{comment}
    \item $C^-$ undergoes a regular homotopy in the complement of $C^+$ (which remains fixed), during which two arcs of $C^-$ change crossing while the rest of $C^-$ is fixed.\label{it:crossing}
\end{enumerate}

We inductively describe how $F_1^{\pm}\subset (-1,0]\times D^3$ extends to $F_2^{\pm}\subset [-1,2]\times D^3$, and hence how $X_1$ extends to $X_2$, according to when these moves occur in Figs. \ref{fig:movie}.1 to \ref{fig:movie-last}.
Suppose $t_0\in [0,1]$ is such that $F_{t_0}$ has been defined by applying the procedure according to Figs. \ref{fig:movie}.1 to \ref{fig:movie}.i for some $i\in \{1,\ldots, \ref{fig:movie-secondlast}\}$. Suppose the next figure (Fig. \ref{fig:movie}.(i+1)) is obtained from Fig. \ref{fig:movie}.i by performing move:

\ref{it:reg-plus}
This defines a regular homotopy of $C^+:=\d F^+_{t_0}$ to a knot $\hat{C}^+$, say, in the complement in a 3-ball of $C^-:=\d F^-_{t_0}$. Let $t_1\in (t_0,1)$ and parameterize the regular  homotopy by $G:S^1\times [t_0,t_1]\looparrowright D^3-\nu(C^-)$, where $G_{t_0}=C^+$ and $G_{t_1}=\hat{C}^+$. Since $C^-$ is fixed over $[t_0,t_1]$, $F_{t_1}^-$ is obtained from $F_{t_0}^-$ by attaching a collar in $[t_0,t_1]\times D^3$. That is,
\[
F_{t_1}^{-}=F^-_{t_0}\underset{\{t_0\}\times C^-}{\cup} ([t_0,t_1]\stimes \d F^-_{t_0})\subset (-1,t_1]\times D^3,
\]
so that $X_{t_1}=X_{t_0}\underset{\{t_0\}\times \d_{t_0}X}{\cup} ([t_0,t_1]\stimes \d_{t_0}X)$, while
\[
    F_{t_1}^+ =  F_{t_0}^+\underset{\{t_0\}\times C^+}{\cup} \bigl([t_0,t_1]\times G(S^1\stimes [t_0,t_1])\bigr)\subset (-1,t_1]\times D^3.
\]
Note that $F_{t_1}^+\subset X_{t_1}$, and $(\d_{t_1} X, \d F^+_{t_1})$ is described with respect to Fig. \ref{fig:movie}.(i+1) as $(D^3-\nu(C^-), \hat{C}^+)$. Moreover, $X_{t_1}$ is clearly diffeomorphic to $X_{t_0}$ (we will denote this $X_{t_1}\approx X_{t_0}$).

\ref{it:crossing}
This defines a regular homotopy $C^-:=\d F^-_{t_0}$ in the complement of $C^+$ to a knot $\hat{C}^-$, during which a single double point $q$, say, of $C^-$ appears. Let $t_1\in (t_0,1)$ and parameterize the  regular homotopy by $G:S^1\times [t_0,t_1]\looparrowright D^3-\nu(C^+)$, where $G_{t_0}=C^-$ and $G_{t_1}=\hat{C}^-$. Then
\[
F_{t_1}^{+}=F^+_{t_0}\underset{\{t_0\}\times C^+}{\cup} ([t_0,t_1]\stimes \d F^+_{t_0})\subset (-1,t_1]\times D^3,
\]
while
\[
    F_{t_1}^- =  F_{t_0}^-\underset{\{t_0\}\times C^-}{\cup} \bigl([t_0,t_1]\times G(S^1\stimes [t_0,t_1])\bigr)\subset (-1,t_1]\times D^3.
\]
Note that $F_{t_1}^+\subset X_{t_1}$. The diffeomorphism type of $X_{t_1}$ is described as follows.
\begin{proposition}\label{prop:twohandle}
The 4-manifold $X_{t_1}$ is obtained from  $X_{t_0}$ by attaching a single 0-framed 2-handle to an unknot near the the crossing that changed in the above homotopy as indicated by Fig. \ref{fig:proof-h} which shows how the Kirby diagram for $X_{t_1}$ is obtained from that of $X_{t_0}$.
\end{proposition}
\begin{proof}
Let $B$ be a 3-ball neighborhood of the crossing of $C^-$ in $D^3$ such that on the complement of $B$, the homotopy $G_t$ is fixed  for all $t$. Then $\d_t X-B=\d_1 X - B$ for $t_0\leq t\leq t_1$; on the other hand, Fig. \ref{fig:proof-a} shows how $\d_t X\cap B$ changes across various values of $t\in [t_0,t_1]$ (here, $t_0<b<c<c'<t_1$). Indeed, for $b\leq t\leq c$ we may assume that $\d_t X\cap B$ appears as shown in Fig. \ref{fig:proof-b} (where $b<b_1<b_2<c$).
%
%\renewcommand\thesubfigure{\alph{subfigure}}
%\captionsetup[subfloat]{captionskip=0.2cm}
\begin{figure}[H]
\addtolength{\subfigcapskip}{0.1cm}
\centering
\subfigure[]{
\label{fig:proof-a}
\centering\includegraphics[width=0.65\textwidth]{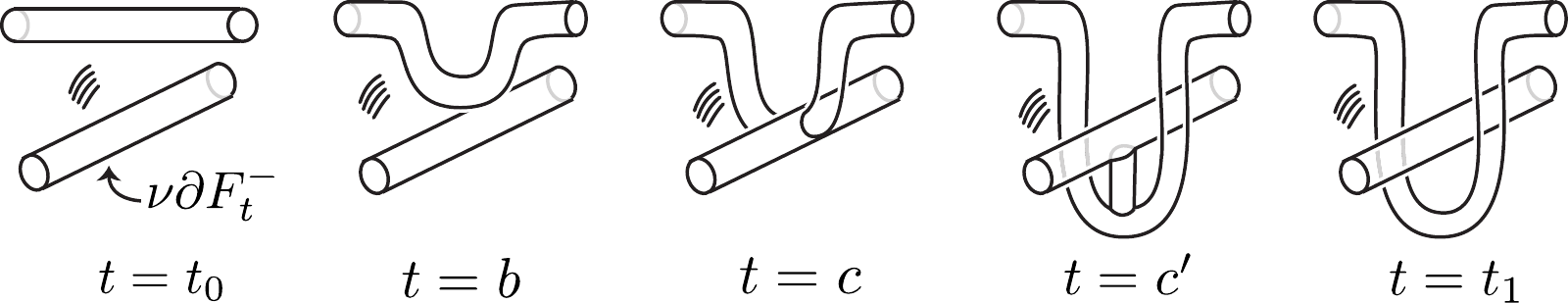}
}
\vspace{0.3cm}
\subfigure[$\d_t X\cap B$]{
\label{fig:proof-b}
\centering
\includegraphics[width=0.45\textwidth]{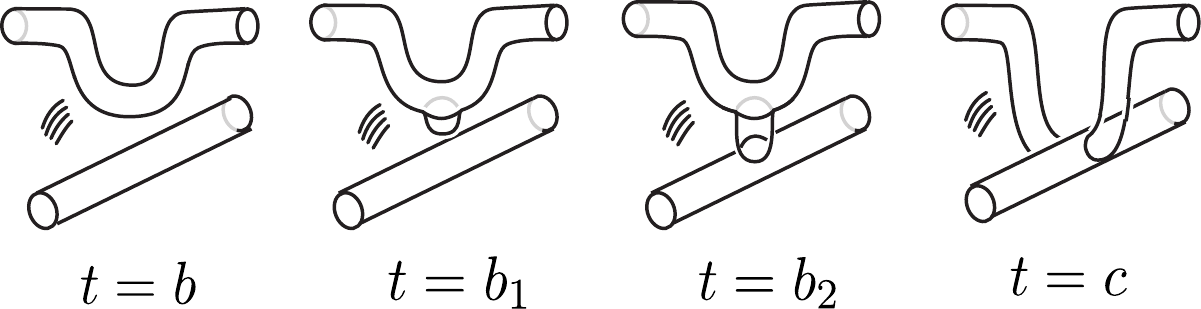}
}
  \caption{}\label{fig:proof1}
\end{figure}
Observe that $X_{b_2}\approx X_c$ is constructed from $X_b$ by first attaching a collar $[b,b_2]\times \d_b X$, then pushing the interior of a 3-ball in $\{b_2\}\times \d_b X$ into $(b, b_2)\times \d_b X$. A schematic is shown in Fig. \ref{fig:proof-c}. This does not change the diffeomorphism type of $X_b\approx X_{t_0}$, so $X_c\approx X_{t_0}$.
\begin{figure}[H]
\centering\includegraphics[width=0.35\textwidth]{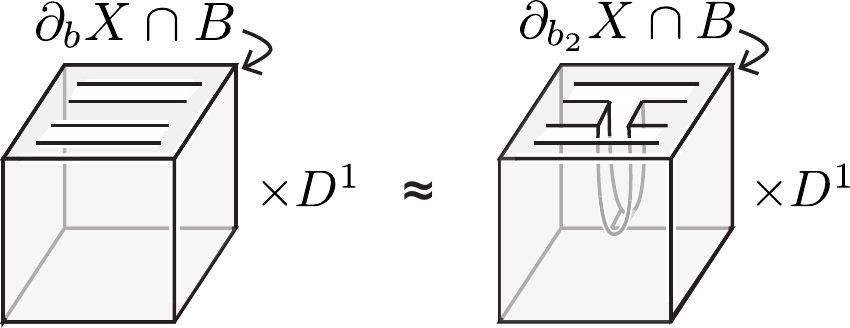}
\caption{Pushing a ball in $\d_{b_2} X$ into the interior.}\label{fig:proof-c}
\end{figure}
Now consider how $X_{t_1}$ is obtained from $X_{c'}\approx X_c$. Note that $\d_{t_1} X- \d_{c'} X$ is a 3-ball $U\subset B$ which we identify with $D^2\times [-1,1]$ so that $D^2\times \{\pm 1\} \subset \nu (\d F^-_{t_1})$ (see Fig. \ref{fig:proof-d}). Hence $X_{t_1}$ is constructed from $X_{c'}$ by attaching a collar $[c',t_1]\times \d_{c'} X$, followed by a family of 3-balls that fill out $U$ in this collar. Let $d\in (c',t_1)$. For $0\leq s\leq 1$, let $D_s=D^2\times [-s,s]\subset U$, and put
\begin{align*}
D=\bigl\{\{d(1-s)+t_1s\} \times D_s:0\leq s\leq 1\bigr\} \subset [d,t_1]\times U.
\end{align*}
See Fig. \ref{fig:proof-e}. Then $X_{t_1} = X_{c'}\cup ([c',t_1]\times \d_{c'} X)\underset{\{d\}\times \d D_0}{\cup} D.$
\begin{figure}[H]
\centering
\subfigure[$\d_{t_1} X\cap B$]{
\centering\includegraphics[width=0.2\textwidth]{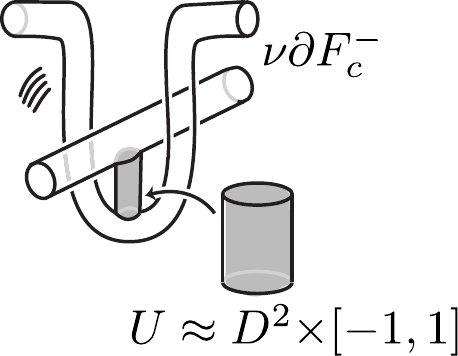}
\label{fig:proof-d}
}
\hskip 1cm
\subfigure[The family $D$ of 3-balls in {\color{white}............. .......}${[}c',t_1{]}\times U\subset X_{t_1}$.]{
\centering\includegraphics[width=0.45\textwidth]{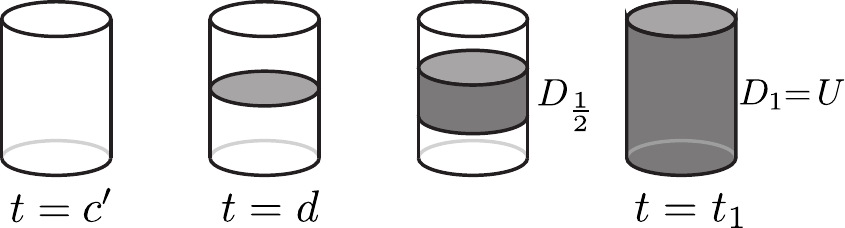}
\label{fig:proof-e}
}
\caption{}\label{fig:proof2}
\end{figure}
Now, $D\approx D^1\times (D^2\times D^1)$ may be viewed as a 2-handle with core $\{d\}\times D_0$, attached to $\widehat{X}_{t_1}:=X_{c'}\cup [c',t_1]\times \d_{c'} X\approx X_{t_0}$ along the circle $K_0=\{d\}\times \d D_0$ shown in Fig. \ref{fig:proof-g}. As a pushoff of $K_0$ in $\d \widehat{X}_{t_1}\cap D$ is of the form $\{d+\eps\} \times \d D_0\subset \{d+\eps\}\times \d_{d+\eps} X_{c'}$ (for some small $\eps>0$), the 2-handle is zero-framed.

\begin{figure}[H]
\centering
\includegraphics[width=\textwidth]{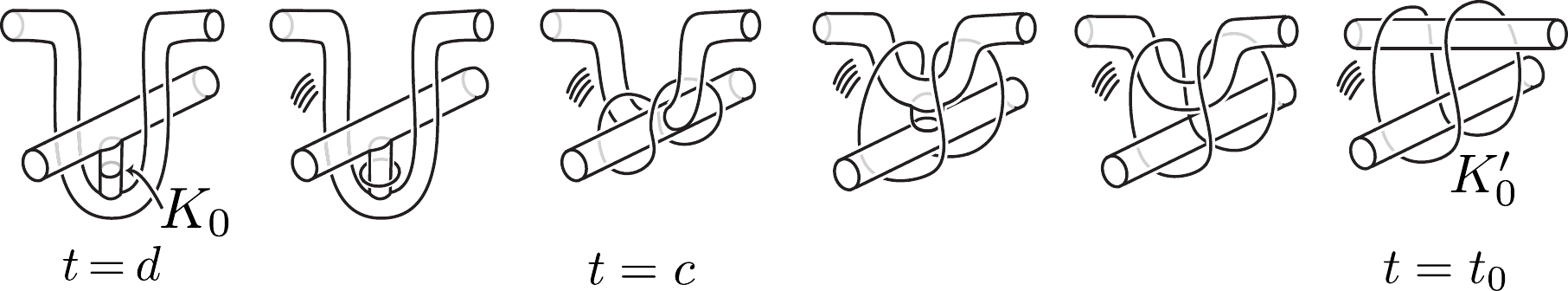}
\caption{$\d_t X\cap B$}\label{fig:proof-g}
\end{figure}

The unknot $K_0$ is isotopic in $X_{t_1}$ to each of the knots in $\d_t X$ (for various values of $t\in [t_1,d]$) shown in Fig. \ref{fig:proof-g}. Thus $X_{t_1} \approx X_{t_0}\cup \{\text{2-handle}\}$, where the 2-handle is zero-framed and attached to a knot of the form $K_0'$ in $\d_{t_1} X\subset \{t_1\}\times D^3$ near the crossing of $C^-=\d F^-_{t_1}$; the proposition follows.\QEDA\end{proof}

\begin{figure}[H]
\centering\includegraphics[width=0.3\textwidth]{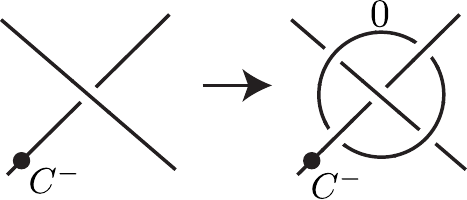}
\caption{The effect a crossing change of $C^-$ has on the Kirby diagram for $X_t$.}\label{fig:proof-h}
\end{figure}

Exhausting the above inductive process extends $F_0^{\pm}$ to $F_r^{\pm}$ for some $r\in (0,1)$ so that $F_{r}^{\pm}$ are properly immersed 2-discs in $(-1,r]\times D^3$ with unknotted boundaries $\d F_{r}^{\pm}$, appearing as $C^{\pm}$ in Fig. \ref{fig:movie-last}, respectively. Moreover, $F^+_r$ is properly immersed in $X_r$, which has Kirby diagram as shown in Fig. \ref{fig:kirby-label}.

For each $1\leq i\leq 5$, let $k_i\in (0,1)$ be such that  $\d F^+_{k_i}\subset \d_{k_i} X$ is the singular (undotted) knot in Fig. \ref{fig:movie} on which the double points $\{p_i^+,p_i^-\}$ appear (so that $k_{i}<k_{i+1}$). For example, $\d_{k_1} X$ appears in Fig. \ref{fig:movie}.3. Let $l_1<k_1$ be such that $\d_{l_1} X$ appears in Fig. \ref{fig:movie}.2; that is, just ``below'' the slice in which the double points $\{p_1^+,p_1^{-}\}$ appear. Similarly let $l_2,\ldots,l_5\in (l_1,1)$ be such that for each $2\leq i\leq 5$, $\d X_{l_i}=\d X_{k_i}$ and $\d F_{k_i}^+$ differs from $\d F_{l_i}^+$ only by the double points $\{p_i^+,p_i^-\}$ having formed. It follows from Proposition \ref{prop:twohandle} that $X_r$ has five 2-handles $h_1,\ldots, h_5$, where $h_i$ is attached to  $\d_{l_i} X$ along the knot $L_i$ shown in Fig. \ref{fig:kirby-label}.

Let $\hat{D}_1^+$ denote the obvious 2-disc bound by the unknot $C^+=\d F_r^+$ shown in Fig. \ref{fig:movie-last}, and let $D_1^+$ be the 2-disc in $[1,2)\times D^3$ obtained by pushing the interior of the 2-disc $\{1\}\times \hat{D}_2^+$ into $(1,2)\times D^3$. We obtain $F_2^{\pm}$ from $F_r^{\pm}$ by attaching collars and, for $F^-$,  capping off with this disc:
\begin{align*}
F_2^+ &= F_r^{-}\underset{\{r\}\times \d F_r^{-}}{\cup} ([r,2]\times \d F_r^{\pm})\subset (-1,2]\times D^3\\
F_2^{+} &= F_r^{+}\underset{\{r\}\times \d F_r^{+}}{\cup} ([r,1]\times \d F_r^{\pm}) \underset{\{1\}\times \d F_r^{\pm}}\cup D_1^+.
\end{align*}
Thus $F_2^+$ is an smoothly immersed 2-sphere in the interior of the smooth 4-manifold $X_2\approx X_r$, and $X_2$ has the Kirby diagram of Fig. \ref{fig:kirby-label}. Put $F^{\pm}:=F_2^{\pm}$ and $X:=X_2$. Note that the 4-manifold $X\subset [-1,2]\times D^3$ inherits a canonical orientation from the 4-ball, though this will not be needed.

It remains to make $F^{+}$ self-transverse.   We see from Fig. \ref{fig:movie} that for each pair of double points $\{p_i^+,p_i^-\}$ of $F^+$ ($1\leq i\leq 5$), positive and negative arcs $\alpha_i$ and $\beta_i$, respectively, may be chosen in $\d_{k_i} F^+$ so that the Whitney circle $\gamma_i=\alpha_i\cup \beta_i$ can be seen entirely in  $\d_{k_i} X\subset \{k_i\}\times D^3$. We now describe a small isotopy of $F^+$ that preserves $\gamma_i$ and has support on a neighborhood of $\alpha_i$. For $\eps>0$ sufficiently small we may identify $(X,F^+)\cap ([k_i-\eps,k_i+\eps]\times D^3) = ([-1,1]\times \d X_{k_i},[-1,1]\times\d F^+_{k_i})$, where $(\{0\}\times \d_{k_i} X,\{0\}\times \d F^+_{k_i}) = (\d_{k_i} X,\d F^+_{k_i})$.
%
%\renewcommand{\figurename}{}
%\captionsetup[figure]{labelformat=empty}
%
\begin{figure}[h]
\centerline{\includegraphics[height=0.33\textwidth]{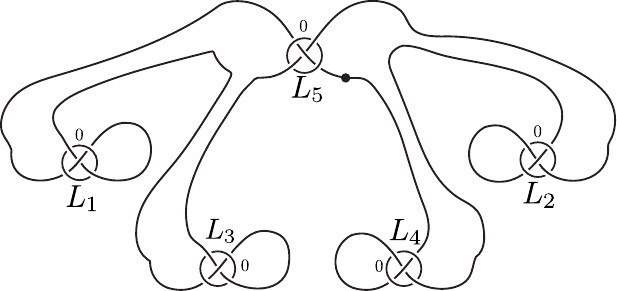}}
\caption{}\label{fig:kirby-label}
%\vskip -0.3cm
\end{figure}
%\renewcommand{\figurename}{Figure}
%\captionsetup[figure]{labelformat=simple}
A neighborhood ($\approx [-1,1] \times D^3$) of the positive arc $\alpha_i$ in $[-1,1]\times\d_{k_i} X$ is shown in Fig. \ref{fig:crossarcs}, appearing as an $[-1,1]$-family of 3-balls. Choose neighborhoods $N(\alpha_i)$ and $N(\beta_i)$ of $\alpha_i$ and $\beta_i$ in $\d F^+_{k_i}$, respectively. Then $U_i':=(-1,1)\times N(\alpha_i)$ and $V_i:=(-1,1)\times N(\beta_i)\subset [-1,1]\times \d X_{k_i}$ define neighborhoods of $\alpha_i$ and $\beta_i$ in $[-1,1]\times \d F^+_{k_i}\subset F^+$, respectively. We describe an isotopy of $F^+$ that has support on $U_i'$ as follows. Let $D$ be a 2-disc neighborhood of $\alpha_i$ in $\int U_i'$, and isotope $U_i'$ by pushing $D$ into $\0\times \d_{k_i} X$, as shown in  Fig. \ref{fig:crossnbds}, to obtain a new neighborhood $U_i$ of $\alpha_i$. Note that this isotopy may be done so that a collar of $U_i'$ is fixed. Now as the arc $\beta_i\subset V_i$ intersects the 2-disc $D\subset U_i$ transversely at $p_i^{\pm}$ in $\0\times \d_{k_i} X$ (see Fig. \subref{fig:crossnbds}) we have that in the isotoped 2-sphere, which we again call $F^+$, the two sheets $U_i$ and $V_i$ intersect transversely at $p_i^{\pm}$ in $[-1,1]\times \d_{k_i} X\subset X$. Henceforth we take $U_i$ and $V_i$ to be the positive and negatives sheets, respectively,  for the pair $\{p_i^+, p_i^-\}$.
\begin{figure}[h]
\centering
\subfigure[]{
\centering
\includegraphics[width=0.72\textwidth]{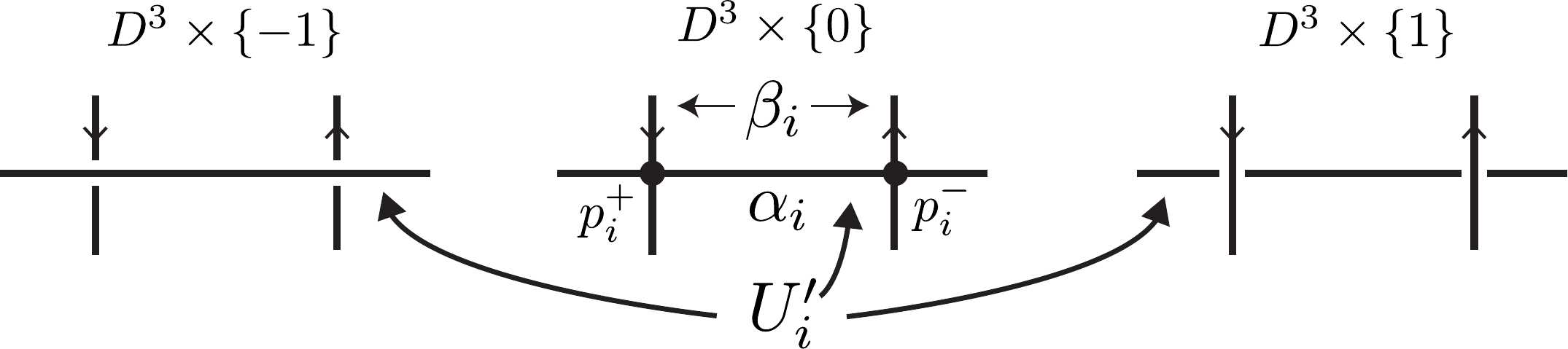}
\label{fig:crossarcs}
}
\vspace{0.3cm}
\subfigure[]{
\includegraphics[width=0.64\textwidth]{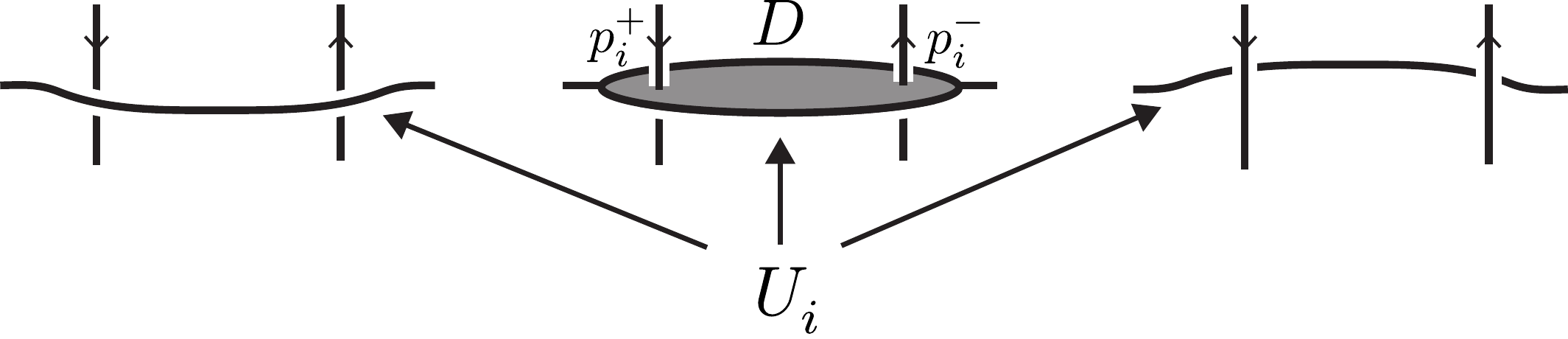}
\label{fig:crossnbds}
}
\caption{}\label{fig:proof3}
\end{figure}
By the construction of $F^{\pm}$, we may choose a link homotopy representative of the link map $f:S^2_+\cup S^2_-\to S^4$ described in \cite{Ki2} Section 6 (with $\sigma(f)=(0,t^2-4t+3)$) such that $f(S^2_+)=F^+$ is properly immersed in $S^4-\nu f(S^2_-)=X$, and $\mu(f^+)=0$. The above modifications do not change the appearance of $F^+$ in Fig. \ref{fig:movie} (as this only shows $\d F_t^+$ for finitely many values of $t$) but give us the following.
\begin{proposition}\label{prop:framing}
If $W_i$ is an embedded 2-disc in $\d_{k_i} X\subset \{k_i\}\times D^3$ with boundary $\gamma_i$, then $W_i$ is a framed Whitney disc for the pair of double points $\{p_i^+,p_i^-\}$ of the immersion $f^+:S^2\to X$.
\end{proposition}
\addvspace{0.5cm}To see this, define a correct framing for $W_i$ as follows. Using the identification above, view the collar $[-1,1]\times \d_{k_i} X\subset \R\times \R^3$. Choose an orientation for $\d W_i$, and let $\v$ denote the normal vector field to $W_i$ in $\R^3$ pointing outward from the side from which the oriented boundary $\d W_i$ runs counterclockwise. Then $\v_1:=(0,\v)$ and $\v_2:=(1,\bo)\in \R\times \R^3$ are linearly independent normal vector fields to $W_i\subset \d_{k_i} X$. Furthermore, by the construction of $U_i$ and $V_i$, we see that $\v_1$ is tangent to $U_i$ along $\alpha_i$ and normal to $V_i$ along $\beta_i$, while $\v_2$ is normal to $U_i$ along $\alpha_i$ and tangent to $V_i$ along  $\beta_i$. Thus $\{\v_1,\v_2\}$ is a correct framing of $W_i$.\QEDA

\section{Computing $\tau(f)$}

We proceed to show that the example $f^+:S^2\to X$ constructed in Section \ref{sec:example} has $\tau(f^+)\neq 0$. Let $t$ denote the homotopy class of a meridian of the dotted component in the Kirby diagram for $X$ (Fig.  \ref{fig:kirby-label}) so that $\pi_1(X)=\Z\langle t\rangle$ and, as described in Section \ref{sec:prelim}, $\tau(f^+)$ lies in $\Z[s^{\pm 1},t^{\pm 1}]/\cR$. We claim that $\Phi(\tau(f^+))=1+t\in \Z_2[\Z_2\times \Z_2]$. First we must verify that $\Phi$ is well-defined on the relations $\cR$.

\begin{lemma}\label{prop:lambda-comp}
For any $A\in \pi_2(X)$, $\Phi(s^k\lambda(F,A))=\Phi(\omega_2(A)s^k)=0$ for all $k\in \Z$.
\end{lemma}

\begin{proposition}\label{prop:phi-homo}
The homomorphism $\Phi$ descends to a surjective homomorphism
\begingroup
\belowdisplayskip=-10pt plus 0pt minus 0cm
\[
    \Z[\Z\times \Z]/\cR \to \Z_2[\Z_2\times\Z_2].
\]
\endgroup
\end{proposition}
\begin{proof}
By Lemma \ref{prop:lambda-comp}, $\Phi$ preserves relation \eqref{reln4}. We check that $\Phi$ is well-defined on the relations \eqref{reln1}-\eqref{reln3}. For $k,l\in \Z$ we have:
\begin{align*}
\Phi(s^kt^k-s^k) &= \begin{cases} 1-1=0 \; \text{ if $k=0\mod 2$,}\\
t-t = 0\; \text{ otherwise};\end{cases}\\
\Phi(s^kt^l+s^lt^k) &= \begin{cases} 1+1=0 \; \text{ if $k=0=l\mod 2$,}\\
t+t = 0\; \text{ otherwise};\end{cases}\\
\text{ and}\\
\Phi(s^kt^l+s^{-k}t^{l-k}) &= \begin{cases} 1+1=0 \; \text{ if $k=0=l\mod 2$}\\
t+t = 0\; \text{ otherwise}.\end{cases}
\end{align*}
Thus $\Phi$ descends to $\Z[\Z\times \Z]/\cR$. Since $1=\Phi(1)$ and $t=\Phi(t)$, the induced map  is surjective.\QEDA\end{proof}

To prove Lemma \ref{prop:lambda-comp} we need the following two propositions, whose proofs are deferred to Section \ref{sec:proofs}, and a calculation of certain Wall intersections.

\begin{proposition}\label{prop:w2}
For any 2-sphere $A\in \pi_2(X)$ we have $\omega_2(A)=0$.
\end{proposition}\vspace{-0.3cm}

\begin{proposition}\label{prop:pi2}
Let $Y$ be an orientable 4-manifold with a handle decomposition consisting of one 0-handle, one 1-handle and $n$ 2-handles. Denote the cores of these handles by $e^0$, $e^1$ and $\{e_1^2,\ldots, e_n^2\}$, respectively, and let $L_i= \d e_i^2$  for each $1\leq i\leq n$. Suppose also that $\pi_1(Y) \cong \Z$. Then
\begin{enumerate}[label=\emph{(\roman{*})}, ref=(\roman{*})]
\item\label{pi2:2} $\pi_2(Y) \cong (\Z[\Z])^n$.
\item\label{pi2:3} Moreover, suppose that for each $1\leq i\leq n$, a (continuous) 2-sphere $A_i$ in $Y$ is constructed by attaching $e_i^2$ to a (continuous) 2-disc in $Y$ along $L_i$. Then the the 2-spheres $A_1,\ldots, A_n$ represent a $\Z[\Z]$-basis for $\pi_2(Y)$.
\end{enumerate}
\vspace{-0.5cm}
\end{proposition}
\begin{proof}[Proof of Lemma \ref{prop:lambda-comp}]
By Proposition \ref{prop:pi2}\label{pi2:2} we have $\pi_2(X)\cong (\Z[\Z])^{5}$. We construct five properly immersed 2-spheres $A_1,A_2,\ldots, A_5$ in $X$ satisfying Proposition \ref{prop:pi2}\ref{pi2:3} and compute their Wall intersections with $F^+$.  Recall from Section \ref{sec:example} that for each $1\leq i\leq 5$, the (zero-framed) 2-handle $h_i$ of $X$ attaches along the knot $L_i$ in $\d_{l_i} X\subset \{l_i\}\times D^3$. The attaching circle $L_i$ bounds an embedded  punctured torus $\hat{T}_i$ in $\d_{l_i} X$; for example, the punctured torus $\hat{T}_1\subset \d_{l_1} X$ is pictured in  Fig. \ref{fig:ver2-kirby-torus}.

\begin{figure}[h]
\centerline{\includegraphics[width=0.5\textwidth]{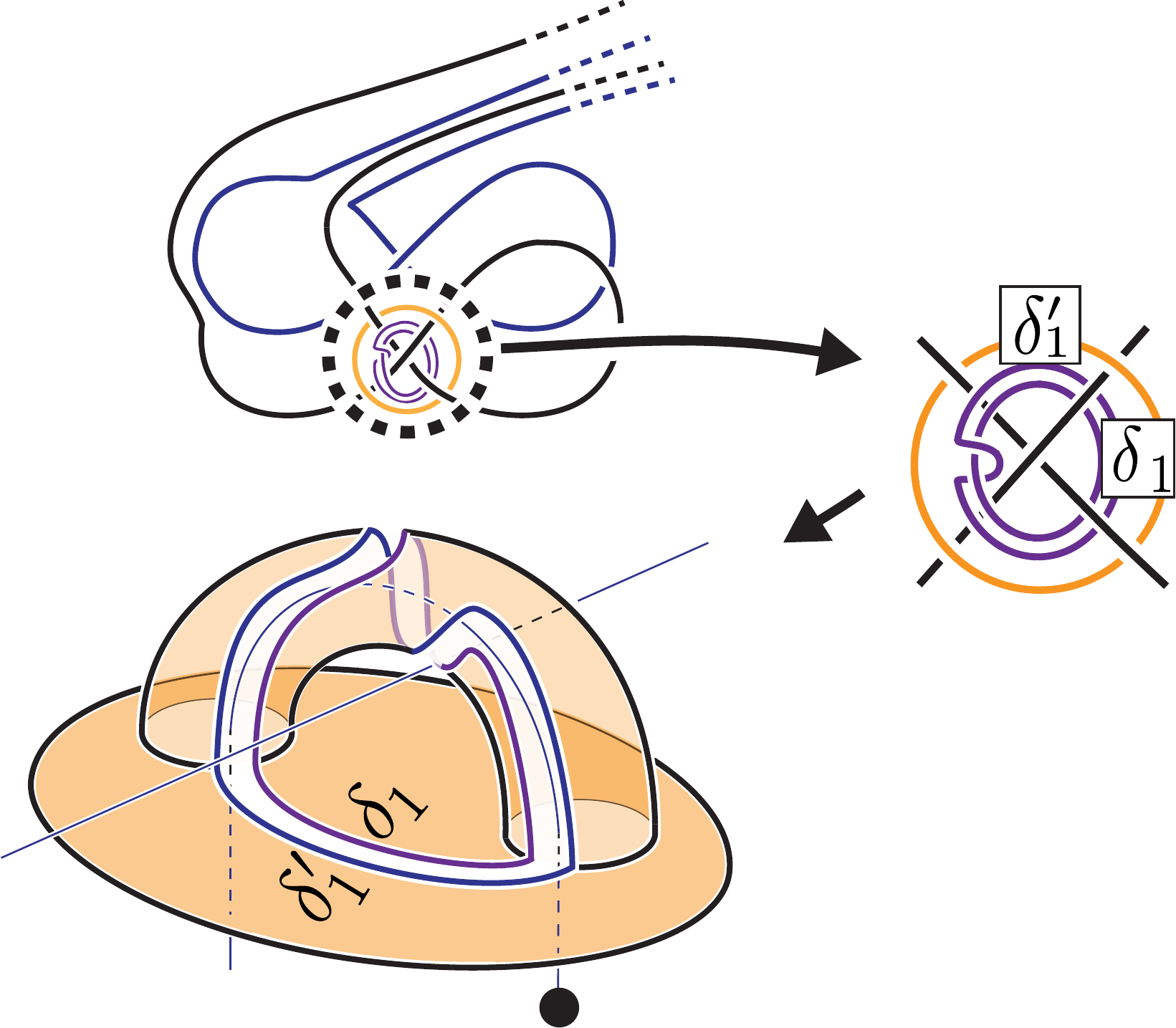}}
\caption{$\d_{l_1} X$}\label{fig:ver2-kirby-torus}
\end{figure}
There is a circle $\delta_i\subset \hat{T}_i$ and (its pushoff) $\delta_i'$ in $\hat{T}_i$ which bound the embedded 2-discs $D_i$ and $D_i'\subset D^3$, respectively, which appear in Fig. \ref{fig:discs}. For example, $D_1$ and $D_1'$ are seen in Figs. \ref{fig:ver2-D1} and \ref{fig:ver2-D1par}, respectively. Let $B_i$ denote the annulus in $\widehat{T}_i$ bound by $\delta_i$ and $\delta_i'$. By the construction of $X$ and $F^+$ we may choose $s_i<l_i$ such that $(X_{l_i}, F_{l_i}^+)\cap ([s_i,l_i]\times D^3) = ([s_i,l_i]\times \d_{l_i}X, [s_i,l_i]\times \d F^+_{l_i})$. We  surger $\widehat{T}_i$ in $[s_i,l_i]\times D^3$ to produce a properly immersed 2-disc in $X_{l_i}$ as follows. Let $s_i'\in(s_i,l_i)$. The circles $\{l_i\}\times \delta_i$ and $\{l_i\}\times \delta_i'$ in the punctured torus $\{l_i\}\times \hat{T}_i\subset \{l_i\}\times \d_{l_i} X$ bound the 2-discs
\begin{align*}
    \hat{D}_i &= (\{s_i\}\times D_i)\underset{\{s_i\}\times\delta_i}{\cup} ([s_i,l_i]\times\delta_i) \text{ and}\\
    \hat{D}_i' &= (\{s_i'\}\times D_i')\underset{\{s_i'\}\times\delta_i'}{\cup} ([s_i',l_i]\times\delta_i'),
\end{align*}
respectively. The 2-disc $\hat{D}_1$, for example, can be seen as first $\{s_1\}\times D_1$ in $\d_{s_1} X = \{s_1\}\times \d_{l_1} X$ (Fig. \ref{fig:ver2-D1}), and then subsequently as $\t \times \delta_1$ in $\d_{t} X=\{t\}\times \d_{l_1} X$ for $s_1\leq t\leq l_1$. The other discs $\hat{D}_i$ (or $\hat{D}_i'$) appear similarly. (Note that although the interiors of these discs are mutually disjoint, for each $i$ the discs $\hat{D}_i$ and $\hat{D}_i'$ intersect (transversely) at a single point in $\d_{s_i'} X$ because $\delta_i$ and $\delta_i'$ link once.) Now, the pair of circles $\{l_i\}\times (\delta_i\cup \delta_i')$ bound the annulus $\{l_i\}\times B_i\subset \{l_i\}\times \hat{T}_i$. Thus, as $\delta_i$ is an  essential loop in $\hat{T}_i$, by removing this annulus from $\{l_i\}\times \hat{T}_i$ and attaching $(\hat{D}_i\cup \hat{D}_i')$ we obtain an immersed 2-disc $\hat{A}_i$ in $[s_i,l_i]\times \d_{l_i} X \subset X_{l_i}$ with boundary $\{l_i\}\times L_i=L_i$. That is, let
\[
\hat{A}_i = \bigl(\{l_i\}\times (\hat{T}_i- \int B_i)\bigr) \underset{\{l_i\}\times (\delta_i\cup \delta_i')}{\cup} (\hat{D}_i\cup \hat{D}_i').
\]
Let $e_i$ denote the core of the 2-handle $h_i$ of $X$ (which attaches to $L_i$ in $\{l_i\}\times \d_{l_i} X$). We use this to cap off $\hat{A}_i$:
\[
A_i = \hat{A}_i \underset{L_i}{\cup} e_i.
\]
Then $A_i$ is an immersed 2-sphere in the interior of $X$ constructed as per Proposition \ref{prop:pi2}\ref{pi2:3}. Thus  $A_1,A_2,\ldots, A_5$ represent a $\Z[\Z]$-basis for $\pi_2(X)\cong (\Z[\Z])^5$.

Let $\tl$ denote the composition
\[
    \tl:\pi_2(X)\times \pi_2(X)\xrightarrow{\lambda(\cdot, \cdot)} \Z[t^{\pm 1}] \xrightarrow{\Phi} \Z_2\langle t\rangle,
\]
where $\Z_2\langle t\rangle = \langle 1, t|\,t^2=1\rangle$ written multiplicatively, and note that since $\Phi(s^kt^n)=\Phi(s^kt^{\bar{n}})$ for any $k,n\in \Z$, we have $\Phi(s^k\lambda(C,D))=\Phi(s^k\tl(C,D))$ for all $C, D\in\pi_2(X)$, $k\in \Z$.
In computing $\tl$ we may choose orientations arbitrarily. Observe also that a point in a slice $\d_t X$ separated from $\d F_t^-$ (the dotted component in one of the subfigures of Fig. \ref{fig:movie}) by a 3-ball can be joined to any other such point in another slice $\d X_{t'}$ ($t<t'$) by a path in $[t,t']\times D^3$ separated from $F^-$ by a 4-ball, so we may safely vary the basepoint of $X$ to be any  point in any slice away from the dotted component.  By Proposition \ref{prop:w2} we need to show that $\Phi(s^k\tl(F^+,A))=0$ for all $A\in \pi_2(X)$. Let $1\leq i\leq 5$. The only intersections between $F^+$ and $A_i$ occur between $\d_{s_i} F^+$ and $\{s_i\}\times D_i$ in $\d_{s_i} X$, and between $\d_{s_i'} F$ and $\{s_i'\}\times D_i'$ in $\d_{s_i'}X$. These intersections are transverse as can be seen, for example, in Figs. \ref{fig:ver2-D1} and \ref{fig:ver2-D1par} for the case $i=1$. Let $\tl(F^+,D_i)$ (or $\tl(F^+,D_i')$)  denote the contribution to $\tl(F^+,A_i)$ due to intersections between $F^+$ and $\{s_i\}\times D_i$ (or $\{s_i'\}\times D_i'$, resp.) in $\d_{s_i} X$ (or $\d_{s_i'} X$, resp.). By the construction of $A_i$, each intersection point between $F^+$ and $\{s_i\}\times D_i$ is paired with exactly one intersection point between $F^+$ and $\{s_i'\}\times D_i'$. We claim that $\tl(F^+, A_i) = (1+t)\tl(F^+,D_i)$. Choose a basepoint for $A_i$ on $\widehat{T}_i$ as shown Fig. \ref{fig:torus-basepoint}. Suppose a point $x\in F^+\cap (\{s_i\}\times D_i)$ contributes $\Phi(\sign(x)r^{x})$ to $\tl(F^+,D_i)$, where $r_x\in \pi_1(X)=\Z\langle t\rangle $ is represented by a loop that goes from the basepoint of $X$ to the basepoint of $F^+$, along $F^+$ to $x$,  along $A_i$ to its basepoint, then back to the basepoint of $X$. Then there is a point $x'\in F^+\cap (\{s_i'\}\times D_i')$ with opposite sign that contributes $\Phi(-\sign(x)r_{x'})$ to $\tl(F^+,D_i')$, where $r_{x}r_{x'}^{-1}$ is homotopic to a loop in $\d_{l_i} X$ that links once with the dotted component $C^-=\d F^-_{l_i}$. See Fig. \ref{fig:torus-basepoint}. Thus $r_{x'}=t^{\eps} r_{x}\in \pi_1(X)$ for some $\eps\in\{-1,1\}$, so the contribution to $\tl(F^+,A_i)$ made by these two points is $\Phi(\sign(x)(1-t^{\eps}r_x))=(1+t)\Phi(r_x)$. The claim then follows.
%\vspace{-0.5cm}
\begin{figure}[h]
   \centering
   \includegraphics[width=0.5\textwidth]{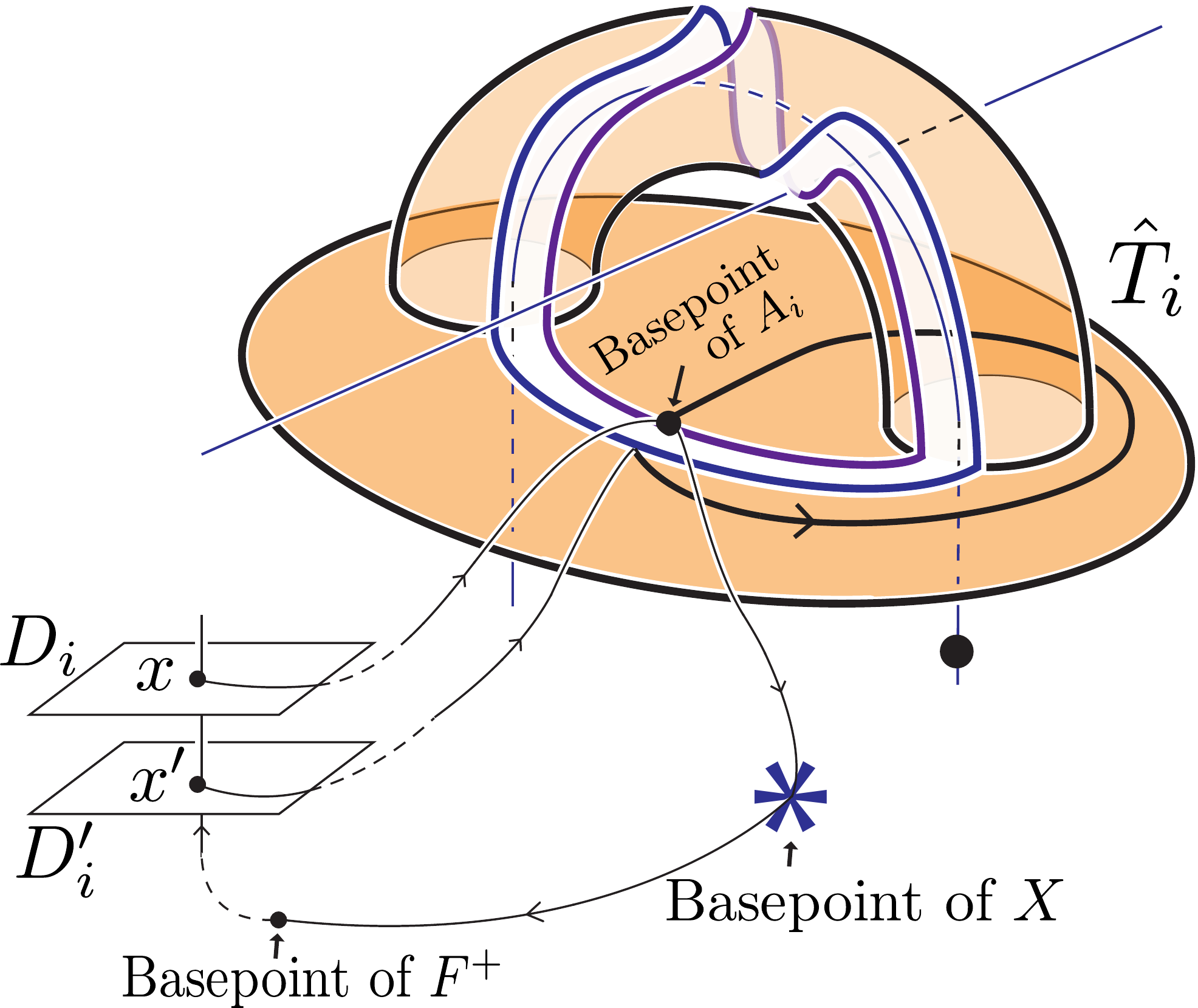}
   \caption{}
   \label{fig:torus-basepoint}
\end{figure}

Referring to Fig. \ref{fig:ver2-D1}, denote the two intersection points between $\d F_{s_1}^+$ and $\{s_1\}\times D_1$ by $\{y,y'\}$. A loop in $\d_{s_1} X$ that goes from $y$ to $y'$ along an arc in $\d F_{s_1}$, then back to $y$ along an arc in $\{s_1\}\times D_1$, links $C^-=\d F_{s_1}^{-}$ once, so we have  $\tl(F^+,D_1)=1+t$. Similarly, from Figs. \ref{fig:ver2-D2}-\ref{fig:ver2-D5} we see that $\tl(F^+,D_i)=1+t$ for $2\leq i\leq 4$, and $\tl(F^+,D_5)=2(1+t)$. We deduce that $\tl(F^+, A_i)=(1+t)^2\equiv 0$ for $1\leq i\leq 4$ and $\tl(F^+,A_5)\equiv 2(1+t)^2\equiv 0 \in \Z_2\langle t\rangle $. Thus $\Phi(s^k\lambda(F^+,A_i))=\Phi(s^k\cdot 0) = 0$ for all $k\in \Z$. Now, if $A\in \pi_2(X)$ then $A=\sum_{i=1}^{5} g_i(t) A_i$ for some $g_i(t)\in \Z[t^{\pm 1}]$, so as $\lambda$ is $\Z[t^{\pm 1}]$-bilinear  we have $\Phi(s^k\lambda(F^+,A))=0\in \Z_2[\Z_2\times \Z_2]$.\QEDA\end{proof}

\section{The Whitney Discs}

Recall from Section \ref{sec:example} that for each $1\leq i\leq 5$, the paired double points $\{p_i^+,p_i^-\}$  of $F^+$ lie in $\d F^+_{k_i}\subset \d_{k_i} X$. An embedded Whitney disc $W_i$ in $\d_{k_i} X\subset \{k_i\}\times D^3$ may be constructed for the pair $\{p_i^+,p_i^-\}$ as shown in Fig. \ref{fig:wds}.i ($1\leq i\leq 5$). As such, by Proposition \ref{prop:framing} we have that each Whitney disc $W_i$ is framed so  $\Phi(I(W_i))$ may be computed as follows.

From Fig. \ref{fig:wds-1} we see that $F^+$ meets $\int W_1$ transversely in exactly one point, $z_1$, say. Since $\d F_{k_1}^+$ and $W_1$ lie in a 3-ball away from $\d F_{k_1}^-=C^-$ (see Fig. \ref{fig:movie}.3) it is apparent that both the primary group element for $W_1$ and the secondary group element corresponding to $z_1$ are trivial in the fundamental group of $X=D^4-\nu F^-$, so $\Phi(I(W_1))=\Phi(\sign(z_1) 1\cdot 1)=1$. From Fig. \ref{fig:wds-2}, we see that $F^+$ intersects $\int W_2$ transversely in 5 points. This includes two obvious pairs of intersection points  whose signs and corresponding secondary elements are the same, so only the remaining intersection point, $z_2$ say, contributes to $\Phi(I(W_2))$. Suppose the secondary element corresponding to $z_2$ is $t^{n_2}\in \pi_1(X)$ for some $n_2\in \Z$. A loop in $\d F^+_{k_2}$ based at $p_2^+$ which leaves along one branch and returns along the other  links the dotted component $C^-=\d F_{k_2}^-$ once, so the primary group element for $W_2$ is $t^{\eps_2}\in \pi_1(X)$ (for some $\eps_2\in\{-1,1\}$). Thus $\Phi(I(W_2))=\Phi(\sign(z_2)s^{\eps_2}\cdot t^{n_2})=t$ (since $\eps_2+\eps_2n_2+n_2\equiv 1+2n_2\equiv 1 \mod 2)$. Using Figs. \ref{fig:wds-2}-\ref{fig:wds-5}, similar considerations show that $\Phi(I(W_3))= 0$ and $\Phi(I(W_4))=\Phi(I(W_5))=t$. Thus
\[
    \Phi(\tau(f^+)) = \sum_{i=1}^5 \Phi(I(W_i))= 1+t+t+t = 1+t\not\equiv 0 \in \Z_2[\Z_2\times \Z_2],
\]
which completes the proof of Theorem \ref{thm:main}.\QEDA

\section{Figures \ref{fig:movie}-\ref{fig:wds}}\label{sec:figs}

% movie of example

\enlargethispage{4cm}

%\vskip -1cm

\begin{figure}[H]
\addtocounter{figure}{1}
\renewcommand{\figurename}{Figs.}
\newcounter{myfig}
\setcounter{myfig}{\value{figure}}
%\addtocounter{myfig}{1}
\renewcommand{\thefigure}{\arabic{myfig}.1-\arabic{myfig}.4}
\centering\includegraphics[width=0.8\textwidth]{oldpk-start-ts.pdf}
\caption{}
\addtocounter{figure}{-1}
\end{figure}
\renewcommand{\figurename}{Fig.}
\manuallabel{fig:movie}{\thefigure}
%\vskip 1.2cm

\newcounter{mycount}
\setcounter{mycount}{5}
\renewcommand{\thefigure}{\arabic{myfig}.\arabic{mycount}}
\begin{figure}[H]
\centering
\includegraphics[width=0.7\textwidth]{oldpk18-ts.pdf}
\caption{}\label{}
\addtocounter{mycount}{1}
\end{figure}

\newpage
\enlargethispage{4cm}
%
%%%%%%%%%%%%%%%%%%%%%%%%%%%%%%%
%%  FIGURE A
%%%%%%%%%%%%%%%%%%%%%%%%%%%%%%%
%\setcounter{mysub}{0}
%\setcounter{figure}{0}

\setlength{\jump}{-0.2cm}
\begin{figure}[H]
\centering
\includegraphics[width=0.8\textwidth]{oldpk-b-1-ts.pdf}
\caption{}\label{}
\addtocounter{mycount}{1}
\end{figure}
\vskip \jump
\begin{figure}[H]
\centering
\includegraphics[width=0.8\textwidth]{oldpk-b-3-ts.pdf}
\caption{}\label{}
\addtocounter{mycount}{1}
\end{figure}
\vskip \jump
\begin{figure}[H]
\centering
\includegraphics[width=0.8\textwidth]{oldpk-b-10-ts.pdf}
\caption{}\label{}
\addtocounter{mycount}{1}
\end{figure}

\newpage
\enlargethispage{4cm}

\setlength{\jump}{-0.2cm}
\begin{figure}[H]
\centering
\includegraphics[width=0.8\textwidth]{oldpk-b-11-ts.pdf}
\caption{}\label{}
\addtocounter{mycount}{1}
\end{figure}
\vskip \jump
\begin{figure}[H]
\centering
\includegraphics[width=0.8\textwidth]{oldpk-b-13-ts.pdf}
\caption{}\label{}
\addtocounter{mycount}{1}
\end{figure}
\vskip \jump
\begin{figure}[H]
\centering
\includegraphics[width=0.8\textwidth]{oldpk-b-16-ts.pdf}
\caption{}\label{}
\addtocounter{mycount}{1}
\end{figure}

\newpage
\enlargethispage{4cm}
\setlength{\jump}{-0.2cm}
\begin{figure}[H]
\centering
\includegraphics[width=0.8\textwidth]{oldpk-b-18-ts.pdf}
\caption{}\label{}
\addtocounter{mycount}{1}
\end{figure}
\vskip \jump
\begin{figure}[H]
\centering
\includegraphics[width=0.8\textwidth]{oldpk-b-22-ts.pdf}
\caption{}\label{}
\addtocounter{mycount}{1}
\end{figure}
\vskip \jump
\begin{figure}[H]
\centering
\includegraphics[width=0.8\textwidth]{oldpk-b-24-ts.pdf}
\caption{}\label{}
\addtocounter{mycount}{1}
\end{figure}

\newpage
\enlargethispage{3cm}

\setlength{\jump}{-0.2cm}
\begin{figure}[H]
\centering
\includegraphics[width=0.8\textwidth]{oldpk-b-27-ts.pdf}
\caption{}\label{}
\manuallabel{fig:movie-secondlast}{\themycount}
\addtocounter{mycount}{1}
\end{figure}
\vskip \jump
\begin{figure}[H]
\centering
\includegraphics[width=0.8\textwidth]{oldpk-b-33-ts.pdf}
\caption{}\label{fig:movie-last}
\addtocounter{mycount}{1}
\end{figure}
\vskip \jump

%%%%%%%%%%%%%%%%%%%%%%%%%%%%%%%%%%%%%%%%%%%%%%%%%%%%%%%%%%%%%%%%%%%% DISCS

\addtocounter{myfig}{1}
\manuallabel{fig:discs}{\themyfig}
\setcounter{mycount}{1}
\begin{figure}[H]
\centering
\includegraphics[width=0.8\textwidth]{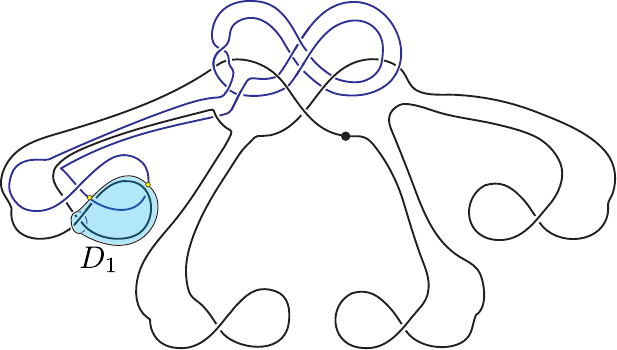}
\caption{$\d_{s_1}X$}\label{fig:ver2-D1}
% was \label{fig:discs}
\addtocounter{mycount}{1}
\end{figure}

\newpage
\enlargethispage{3cm}

\setlength{\jump}{-0.2cm}
\begin{figure}[H]
\centering
\includegraphics[width=0.8\textwidth]{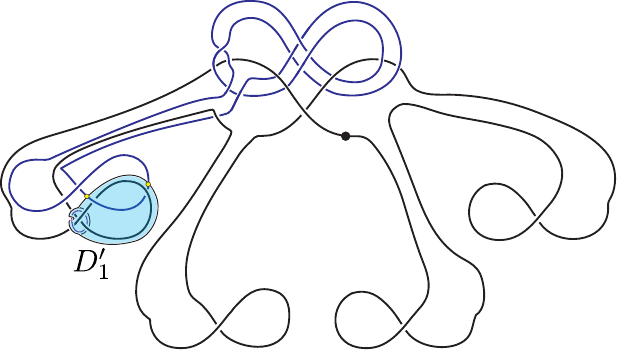}
\caption{}\label{fig:ver2-D1par}
\addtocounter{mycount}{1}
\end{figure}
\vskip \jump
\begin{figure}[H]
\centering
\includegraphics[width=0.8\textwidth]{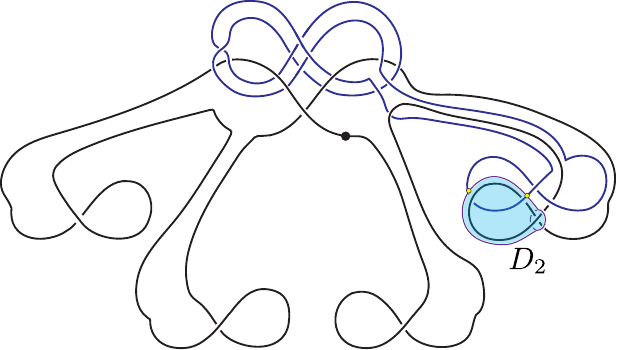}
\caption{}\label{fig:ver2-D2}
\addtocounter{mycount}{1}
\end{figure}
\vskip \jump
\begin{figure}[H]
\centering
\includegraphics[width=0.8\textwidth]{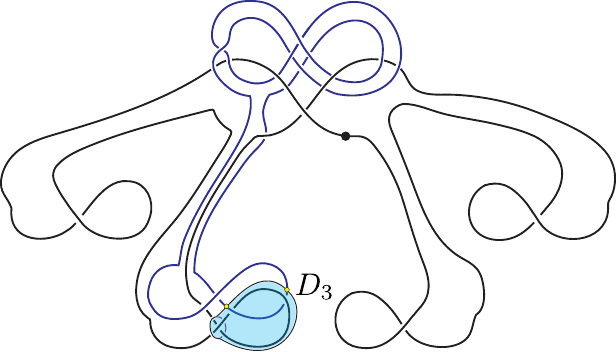}
\caption{}\label{fig:ver2-D2}
\addtocounter{mycount}{1}
\end{figure}

\newpage
\enlargethispage{3cm}

\setlength{\jump}{-0.2cm}

\begin{figure}[H]
\centering
\includegraphics[width=0.8\textwidth]{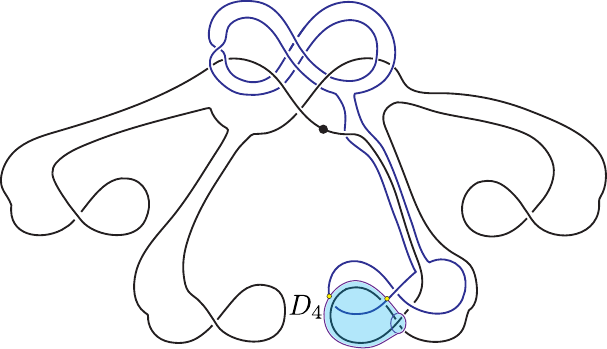}
\caption{}\label{fig:ver2-D4}
\addtocounter{mycount}{1}
\end{figure}
\vskip \jump
\begin{figure}[H]
\centering
\includegraphics[width=0.8\textwidth]{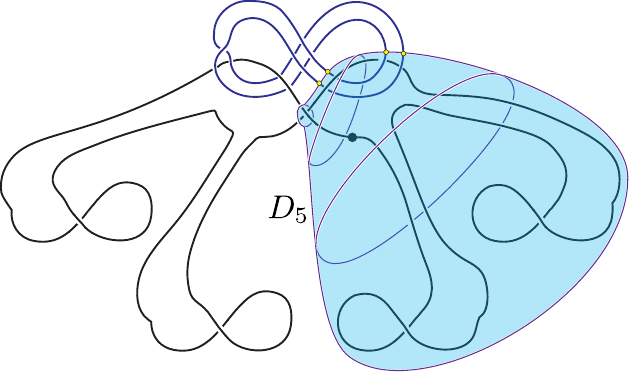}
\caption{}\label{fig:ver2-D5}
\addtocounter{mycount}{1}
\end{figure}

% WHITNEY DISCs %%%%%%%%%%%%%%%%%%%%%%%%%%%%%%%%%%%%%%%%%%%%%%%%%%%%%%
\addtocounter{myfig}{1}
\manuallabel{fig:wds}{\themyfig}
\setcounter{mycount}{1}

\begin{figure}[H]
\centering
\includegraphics[width=0.4\textwidth]{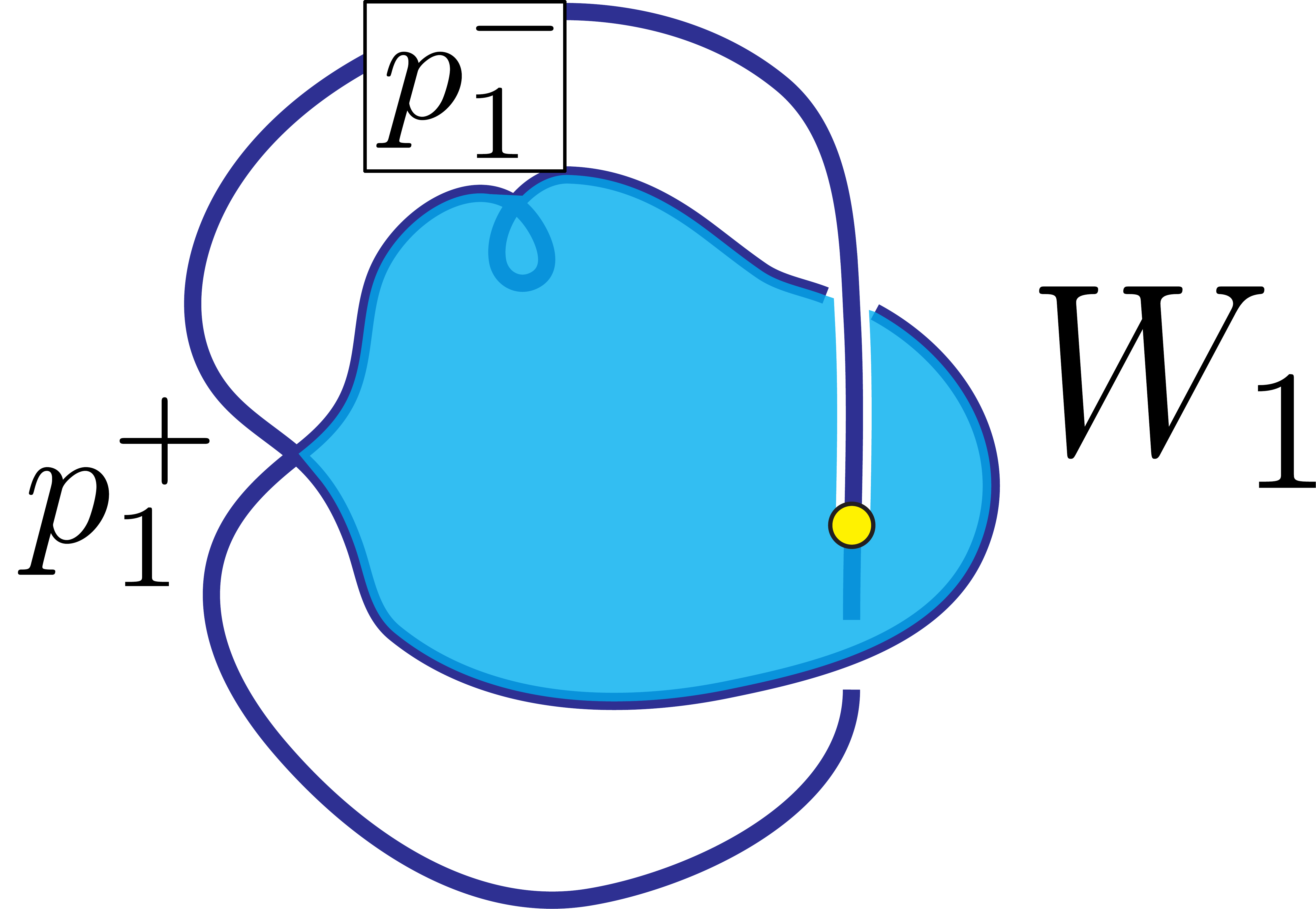}
\caption{}\label{fig:wds-1}
\addtocounter{mycount}{1}
%%%%%%%%%%%%%%%%%%%%%%%%%%%%%%% \label{figs:wds}
\end{figure}

\newpage
\enlargethispage{3cm}

\setlength{\jump}{-0.2cm}
\begin{figure}[H]
\centering
\includegraphics[width=0.82\textwidth]{wd2-ts.pdf}
\caption{}\label{fig:wds-2}
\addtocounter{mycount}{1}
\end{figure}
\vskip \jump
\begin{figure}[H]
\centering
\includegraphics[width=0.82\textwidth]{wd3-ts.pdf}
\caption{}\label{fig:wds-3}
\addtocounter{mycount}{1}
\end{figure}
\vskip \jump
\begin{figure}[H]
\centering
\includegraphics[width=0.82\textwidth]{wd4-ts.pdf}
\caption{}\label{fig:wds-4}
\addtocounter{mycount}{1}
\end{figure}

\newpage

\begin{figure}[H]
\centering
\includegraphics[width=0.82\textwidth]{wd5-ts.pdf}
\caption{}\label{fig:wds-5}
\addtocounter{mycount}{1}
\end{figure}

%\vskip -3cm
%\makeatletter
%\renewcommand{\@evenhead}{\hbox{\hspace{1.3cm}}{\foliofont\thepage}\hspace*{12pt}{\rhfont\leftmark}\hfill}
%\renewcommand{\@oddhead}{\hbox{\hspace{1.2cm}}\hfill{\rhfont\rightmark}\hspace*{12pt}{\foliofont\thepage}}
%\makeatother
\section{Proof of two propositions}\label{sec:proofs}

\begin{proof}[Proof of Proposition \ref{prop:w2}]
\vskip -0.25cm
Represent $A$ by an immersed 2-sphere $S^2\overset{A}{\looparrowright} X$. Since the 4-manifold $X$ has the Kirby diagram shown in Fig. \ref{fig:kirby-label}, in which all the 2-handles are even-framed, we have $\omega_2(TX)=0$ by \cite{GS} Corollary 5.7.2. As $TS^2 \oplus \nu_A = A^\ast(TX)$ and $\omega_1(TS^2)=0=\omega_2(TS^2)$, by \cite{DK} Theorem 10.39 we have
\[
\omega_2(\nu_A) = \omega_2(A^\ast(TX)) = A^\ast\omega_2(TX)=0.
\]
% \omega_1(TS^2)=0 since S^2 orientable; \omega_2(TS^2)=0 since TS^2 is stably trivial
\QEDA\end{proof}

\begin{proof}[Proof of Proposition \ref{prop:pi2}]
Orient the core of each handle of $Y$ and let $Y^{(k)}$ denote the subhandlebody of $Y$ consisting of its handles of index up to and including $k$. Let $t$ denote the homotopy class of the image $e^1/e^0=S^1$ of the 1-handle under the collapsing map $Y^{(1)}\to Y^{(1)}/Y^{(0)}\sim e^1/e^0$  so that $\pi_1(Y)\cong\Z\langle t\rangle$. The corresponding cellular chain complex for  $Y$ is then
\[
    0\to  \overset{n}{\underset{i=1}{\oplus}}\Z\langle e_i^2\rangle \xrightarrow{\d_2} \Z\langle e^1\rangle \xrightarrow{\d_1} \Z\langle e^0\rangle \to 0.
\]
Since $\pi_1(Y)=\Z$, we have $\d_2=0$ and so $H_2(Y)= \ker \d_2/\Im \d_3 \cong \Z^{n}$, and for  each $1\leq i\leq n$ the attaching circle $L_i$ bounds a continuous 2-disc $D_i$ in $Y^{(0)}$. Let $f_i:S^2\to Y^{(2)}$ be a continuous map with image $\hat{A}_i := e_i^2 \underset{L_i}{\cup} D_i$. Upon collapsing $Y=Y^{(2)}$ to a wedge of 2-spheres via $Y^{(2)}\to Y^{(2)}/Y^{(1)} {\sim} \overset{n}{\underset{i=1}{\vee}} (S^2_i=e_i^2/\d e_i^2)$, each 2-sphere $\hat{A}_i$ is mapped precisely to $S^2_i$, so its homology class $[\hat{A}_i]$ maps to the generator $[e_i^2]$ under the isomorphism $H_2(Y)\cong H_2(Y_2/Y_1)\cong C_2(Y)$. Thus $[\hat{A}_1], \ldots, [\hat{A}_n]$ is a $\Z$-basis for $H_2(Y)=\Z^n$. Let $\tY \xrightarrow{p} Y$ be the universal cover and choose (canonically oriented) lifts $\te^k$ in $\tY$ of the oriented cores $e^k$. The induced cellular chain complex of $\tY$ is then
\[
    0\to \overset{n}{\underset{i=1}{\oplus}}\Z[t^{\pm 1}]\langle \te_i^2\rangle \xrightarrow{\td_2} \Z[t^{\pm 1}]\langle \te^1\rangle \xrightarrow{\td_1} \Z[t^{\pm 1}]\langle \te^0\rangle \to 0.
\]
As $\td_1$ is given by $\td_1(\te^1) = (t-1)\te^0$, it is injective and so $\td_2=0$. Hence $H_2(\tY) = \ker \td_2/\Im \td_3 \cong  (\Z[\Z])^{n}$, and $[\te_1],\ldots, [\te_n]$ is a $\Z[t^{\pm 1}]$-basis. Using the isomorphism $p_\ast: \pi_2(\tY)\to \pi_2(Y)$ and the Hurewisc isomorphism  $\rho: \pi_2(\tY)\to H_2(\tY)$ we obtain
\[
    \pi_2(Y)\cong \pi_2(\tY)\cong H_2(\tY) = (\Z[\Z])^n.
\]
This proves part (i).

Now, a map $S^2\to X$ with image $A_i$ as in the hypotheses of part \ref{pi2:3} is homotopic to $f_i$ (as defined above) so we may assume that $A_i=\hat{A_i}$. Let $\tw{L}_i:=\d \te_i$ and let $\tw{D}_i$ be the unique lift of $D_i$ in $\tY$ that has boundary $\d \tw{D}_i=\tw{L}_i$. As $D_i\subset Y^{(0)}$, we have $\tw{D_i}\subset (\tY)^{(0)}$. Thus, as before,
\[
    \tA_i := \te_i^2 \underset{\tw{L}_i}{\cup} \tw{D}_i
\]
is a continuous 2-sphere in $\tY$ such that $[\tA_i]=[\te_i] \in H_2(\tY) \cong C_2(\tY)$, so $\tA_1,\ldots, \tA_n$ is a $\Z[\Z]$-basis for $\pi_2(Y)\overset{p_\ast^{-1}}{\cong} \pi_2(\tY)\overset{\rho^{-1}}{\cong} H_2(\tY)$.  Part \ref{pi2:3} then follows since $p(\tA_i)=A_i$ for each $1\leq i\leq n$.\QEDA\end{proof}

\bibliographystyle{plain}
\bibliography{bib1}
%\vspace{-0.5cm}

%\newpage
%\setlength\oddsidemargin{\storemargin}
\begin{comment}
Figures are to be sequentially numbered in Arabic numerals. The
caption must be placed below the figure. Typeset in 8~pt Times
Roman with baselineskip of 10~pt. Use double spacing between a
caption and the text that follows immediately.

Figures should be referred to in the abbreviated form,
e.g.~``$\ldots$ in Fig.~\ref{fig1}'' or ``$\ldots$ in
Figs.~\ref{fig1} and 2''. Where the word ``Figure'' begins a
sentence, it should be spelt\break
in full.
\end{comment}
%\appendix
%%%%%%%%%%% BIBLIOGRAPHY

%\addcontentsline{toc}{chapter}{References\hfill}
%\addtolength{\oddsidemargin}{-1.2cm}
%\addtolength{\evensidemargin}{-1.3cm}
%\makeatletter
%\renewcommand{\@evenhead}{\hbox{\hspace{1.3cm}}{\foliofont\thepage}\hspace*{12pt}{\rhfont\leftmark}\hfill}
%\renewcommand{\@oddhead}{\hbox{\hspace{1.54cm}}\hfill{\rhfont\rightmark}\hspace*{12pt}{\foliofont\thepage}}
%\makeatother
%\bibliographystyle{plain}

\begin{comment}
\vskip 1cm
\hfill\begin{minipage}{0.95\textwidth}
\bibliography{bib1}
\end{minipage}
\end{comment}
\end{document}